\documentclass[12pt, a4paper]{article}
\usepackage{amsmath}   
\usepackage{comment}
\usepackage{amsfonts}
\usepackage{amssymb}   
\usepackage{epsfig}
\usepackage{graphicx}
\usepackage{amsthm}
\usepackage{newlfont}
\usepackage{dsfont}
\usepackage{mathrsfs}
\usepackage{eufrak}
\usepackage{color}


\newtheorem{thm}{Theorem}[section]

\newtheorem{cor}[thm]{Corollary}
\newtheorem{pro}[thm]{Proposition}

\newtheorem{obs}[thm]{Observation}
\theoremstyle{definition}

\newtheorem{que}[thm]{Question}
\theoremstyle{remark}
\newtheorem*{rem}{Remark}

\graphicspath{{figures/}}
\begin{document}
\title{Characterizing simplex graphs}

\footnotetext[1]{The work is partially supported by the National Natural Science Foundation of China (Grant No. 12071194, 11571155).}

\author{Yan-Ting Xie$\thanks{E-mail address: ~ls\_xieyt@lzu.edu.cn (Y.-T. Xie).}$, ~Shou-Jun Xu$\thanks{Corresponding author. E-mail address: ~shjxu@lzu.edu.cn (S.-J. Xu).}$}

\date{\small School of Mathematics and Statistics, Gansu Center for Applied Mathematics, \\Lanzhou University,
Lanzhou, Gansu 730000, China}

\maketitle
\begin{abstract}
The simplex graph $S(G)$ of a graph $G$ is defined as the graph whose vertices are the cliques of $G$ (including the empty set), with two vertices being adjacent if, as cliques of $G$, they differ in exactly one vertex. Simplex graphs form a subclass of median graphs and include many well-known families of graphs, such as gear graphs, Fibonacci cubes and Lucas cubes. 

In this paper, we characterize simplex graphs from four different perspectives: the first focuses on a graph class associated with downwards-closed sets---namely, the daisy cubes; the second identifies all forbidden partial cube-minors of simplex graphs; the third is from the perspective of the $\Theta$ equivalent classes; and the fourth explores the relationship between the maximum degree and the isometric dimension. Furthermore,  very recently, Betre et al.\ [K. H. Betre, Y. X. Zhang, C. Edmond, Pure simplicial and clique complexes with a fixed number of facets, 2024, arXiv: 2411.12945v1] proved that an abstract simplicial complex (i.e., an independence system) of a finite set can be represented to a clique complex of a graph if and only if it satisfies the Weak Median Property. As a corollary, we rederive this result by using the graph-theoretical method.

\setlength{\baselineskip}{17pt}
{} \vskip 0.1in \noindent%
\end{abstract}
\section{Introduction}
Let $G$ be a graph with vertex set $V(G)$ and edge set $E(G)$. A {\em clique} of $G$ is a subset of $V(G)$ in which each pair of  vertices are adjacent.  The {\em simplex graph} $S(G)$, introduced by Bandelt and van de Vel \cite{bv89}, is the graph whose vertices are the cliques of $G$ (including the empty set), with two vertices being adjacent if, as cliques of $G$, they differ in exactly one vertex (see Fig. \ref{fig:SimplexGraph} for an example).
\begin{figure}[!htbp]
\centering
\setlength{\unitlength}{1mm}
\begin{picture}(90,55)
\put(0,0){\scalebox{0.5}[0.5]{\includegraphics{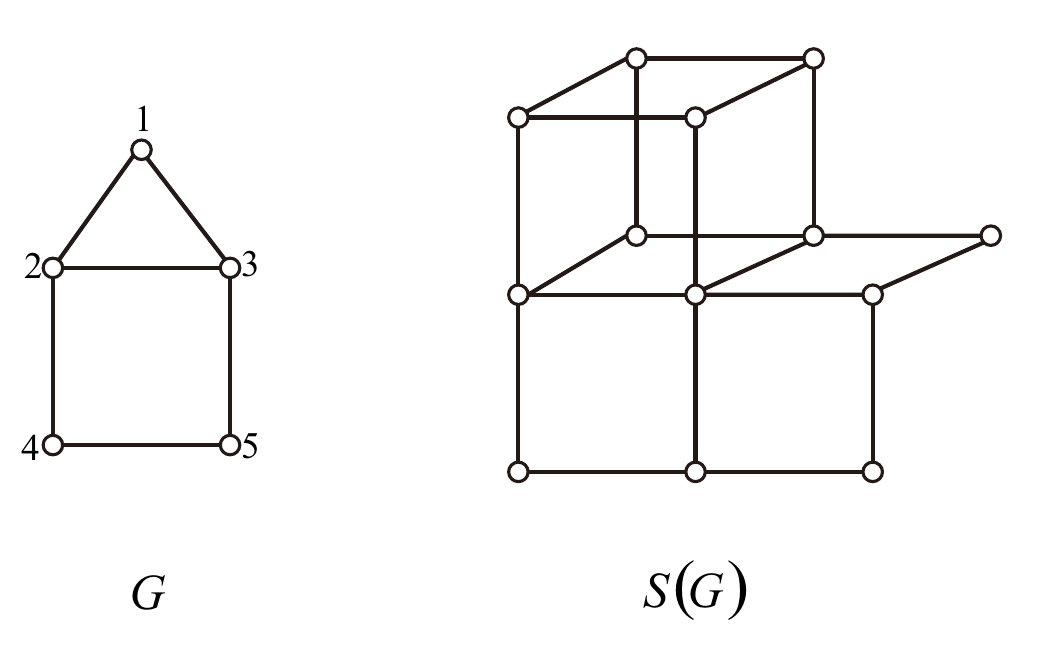}}}
\put(59.5,27){{\fontsize{9pt}{16pt}\selectfont $\emptyset$}}
\put(59.5,42){{\fontsize{9pt}{16pt}\selectfont $\{1\}$}}
\put(69.5,50){{\fontsize{9pt}{16pt}\selectfont $\{1,3\}$}}
\put(49.2,51.2){{\fontsize{9pt}{16pt}\selectfont $\{1,2,3\}$}}
\put(35.5,46){{\fontsize{9pt}{16pt}\selectfont $\{1,2\}$}}
\put(38,29){{\fontsize{9pt}{16pt}\selectfont $\{2\}$}}
\put(45.6,36){{\fontsize{9pt}{16pt}\selectfont $\{2,3\}$}}
\put(69.5,36.5){{\fontsize{9pt}{16pt}\selectfont $\{3\}$}}
\put(81.5,36.5){{\fontsize{9pt}{16pt}\selectfont $\{3,5\}$}}
\put(74.5,27.5){{\fontsize{9pt}{16pt}\selectfont $\{5\}$}}
\put(72.5,12){{\fontsize{9pt}{16pt}\selectfont $\{4,5\}$}}
\put(57,12){{\fontsize{9pt}{16pt}\selectfont $\{4\}$}}
\put(35.5,12){{\fontsize{9pt}{16pt}\selectfont $\{2,4\}$}}
\end{picture}
\caption{Graph $G$ with $V(G)=\{1,2,3,4,5\}$ and its simplex graph $S(G)$. In this example, $S(G)$ is isomorphic to a 5-dimensional Fibonacci cube.}{\label{fig:SimplexGraph}}
\end{figure}
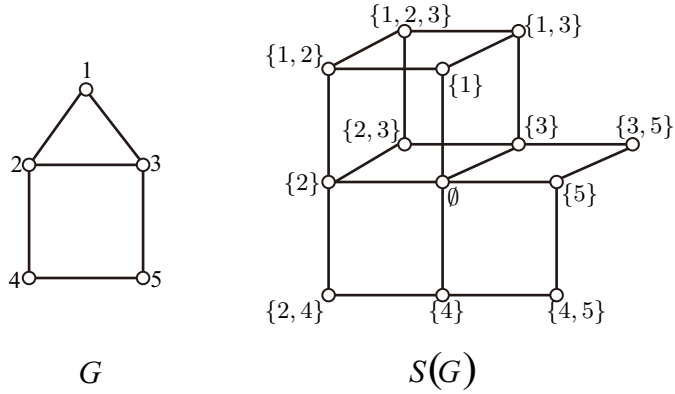

Let $B^n$ be the set of binary strings of length $n$, i.e., $B^n:=\{v=a_1a_2\cdots a_n|a_i\in\{0,1\},\mbox{ for }i=1, 2, \cdots, n\}$. A {\em hypercube of dimension $n$} (or {\em $n$-cube} for short) $Q_n$ is a graph with vertex set $B^n$ and two vertices are adjacent if their corresponding binary strings differ in exactly one coordinate. By convention, we denote the {\em complete graph}, the {\em path} and the {\em cycle} of order $n$ by $K_n$, $P_n$, and $C_n$, respectively.

It is known that many interesting graph classes are simplex graphs. For example: $S(K_n)(\cong Q_n)$, the {\em gear graph} $S(C_n)$, also called the {\em bipartite wheel} \cite{k09}, the {\em Fibonacci cube} $S(\bar{P_n})$ \cite{k13}, and the {\em Lucas cube} $S(\bar{C_n})$ \cite{mcs01,t13}, where $\bar{G}$ represents the complement of the graph $G$.

Let $G$ be a graph with vertex set $V(G)$. For $u, v\in V(G)$, the {\em distance} $d_{G}(u,v)$ (we will drop the subscript $G$ here or the subscript (or the superscript) $G$ in other notations if no confusion arises) is the length of a shortest path between $u$ and $v$ in $G$. A subgraph $H$ of $G$ is called {\em isometric} if, for any $u, v\in V(H)$, $d_H(u,v)=d_G(u,v)$.  A graph $G$ is called a {\em partial cube} if it is isomorphic to an isometric subgraph of $Q_n$ for some integer $n$. Thus, for a partial cube $G$, $V(G)$ can be viewed as a subset of $B^n$. Note that the distance $d_{Q_n}(u,v)$ between vertices $u$ and $v$ in $Q_n$ is the {\em Hamming distance}, i.e., the number of coordinates at which the corresponding symbols in the two binary strings $u$ and $v$ of lengths $n$ differ, so is the distance $d_{G}(u,v)$ in a partial cube $G$ by the definition. $G$ is called a {\em median graph} if, for every three different vertices $u, v, w\in V(G)$, there exists exactly one vertex $x\in V(G)$ (possibly $x\in\{u,v,w\}$), called the {\em median} of $u, v, w$, satisfying the following conditions: $d(u,x)+d(x,v)=d(u,v)$, $d(u,x)+d(x,w)=d(u,w)$, $d(v,x)+d(x,w)=d(v,w)$. Note that median graphs form a significant subclass of partial cubes (see \cite{bv87,km99,m78,m80a,m80b,m90,m11,xfx24}, as well as Chapter 12 of the book \cite{hik11}). 

Bandelt and van de Vel  proved the following result:
\begin{pro}{\em\cite{bv89}}\label{pro:SimplexGraphvsMedianGraph}
For any graph $G$,  $S(G)$ is a median graph.
\end{pro}

A {\em partial order} on $B^n$ is defined as follows: $a_1a_2\cdots a_n\leq b_1b_2\cdots b_n$ if $a_i\leq b_i$ holds for $1\leq i\leq n$. Let $X$ be a subset of $B^n$.  The subgraph of $Q_n$ induced by the set $\{v\in B^n|v\leq x\mbox{ for some }x\in X\}$ is called the {\em daisy cube generated by $X$}, denoted by $Q_n(X)$ \cite{km19}. In this paper, we give the first characterization of simplex graphs: precisely the intersection of median graphs and daisy cubes.

Let $G$ be a connected graph. The {\em Djokovi\'c-Winkler relation} (see \cite{d73,w84}) $\Theta_G$ is a binary relation on $E(G)$ defined as follows: For edges $uv, xy$ in $G$, $$uv\,\Theta_G\,xy\iff d(u,x)+d(v,y)\neq d(u,y)+d(v,x).$$ If $G$ is bipartite, there is an equivalent definition of the Djokovi\'c-Winkler relation: 
\begin{align*}
uv\,\Theta_G\,xy\iff d(u,x)=d(v,y) \mbox{ and } d(u,y)=d(v,x).
\end{align*}
Winkler \cite{w84} proved that a graph $G$ is a partial cube if and only if $G$ is bipartite and $\Theta$ is an equivalence relation on $E(G)$.

For an edge $uv$ of a graph $G$, the following sets play critical roles:
\begin{itemize}
\item $W^G_{uv}:=\{w\in V(G)|d(u,w)<d(v,w)\}$,
\item $W^G_{vu}:=\{w\in V(G)|d(v,w)<d(u,w)\}$,
\item $F^G_{uv}:=\{ab\in E(G)|a\in W^G_{uv}\mbox{ and }b\in W^G_{vu}\}$,
\item $U^G_{uv}:=\{w\in V(G)|w\in W^G_{uv}\mbox{ and }w\mbox{ is incident with an edge in }F^G_{uv}\}$,
\item $U^G_{vu}:=\{w\in V(G)|w\in W^G_{vu}\mbox{ and }w\mbox{ is incident with an edge in }F^G_{uv}\}$.
\end{itemize}

If $G$ is a partial cube with an edge $uv$, then $F_{uv}$ is the {\em $\Theta$ equivalent class} (abbreviated, {\em $\Theta$-class}) containing $uv$, and $W_{uv}$ and $W_{vu}$ are complementary convex subsets of $G$, called {\em halfspaces}.  Furthermore, every $\Theta$-class $F_{uv}$ is an edge cutset of $G$. Without emphasizing the edges, the $\Theta$-classes are also denoted by $\theta_1,\theta_2,\cdots$. 

Let $G$ be a partial cube which can be embedded into $Q_n$ isometrically (i.e., $V(G)\subseteq B^n$). Each $\Theta$-class of $G$ corresponds to a coordinate in the binary strings (see \cite{d73}); specifically, $\theta_i=\{ab\in E(G)|a=a_1\cdots a_{i-1}0a_{i+1}\cdots a_n,b=a_1\cdots a_{i-1}1a_{i+1}\cdots a_n\}$ for an integer $1\leq i\leq n$. Furthermore, if $uv\in\theta_i$ and the $i$-th coordinate of $u$ is 0, then $W_{uv}=\{a=a_1a_2\cdots a_n\in V(G)|a_i=0\}$, $W_{vu}=\{a=a_1a_2\cdots a_n\in V(G)|a_i=1\}$, otherwise $W_{uv}=\{a=a_1a_2\cdots a_n\in V(G)|a_i=1\}$, $W_{vu}=\{a=a_1a_2\cdots a_n\in V(G)|a_i=0\}$. An {\em elementary restriction} with respect to $\theta_i$ consists in taking one of the subgraphs of $G-\theta_i$ ($G-\theta_i$ is the graph obtained by deleting all edges in $\theta_i$ from $G$) induced by the corresponding halfspace, which we will denote by $\rho_{i}^0(G)$ and $\rho_{i}^1(G)$, respectively, i.e., $\rho_{i}^0(G)$ (resp. $\rho_{i}^1(G)$) is the subgraph of $G$ induced by the halfspace $\{a=a_1a_2\cdots a_n\in V(G)|a_i=0\}$ (resp. $\{a=a_1a_2\cdots a_n\in V(G)|a_i=1\}$). Now applying twice the elementary restrictions with respect to two $\Theta$-classes $\theta_i$ and $\theta_j$, independently of their orders, we will obtain one of the four (possibly empty) subgraphs induced by the $\rho_{i}^0(G)\cap \rho_{j}^0(G)$,  $\rho_{i}^0(G)\cap \rho_{j}^1(G)$,  $\rho_{i}^1(G)\cap \rho_{j}^0(G)$, and  $\rho_{i}^1(G)\cap \rho_{j}^1(G)$. More generally, a {\em restriction} is a subgraph of $G$ induced by the intersection of a set of (non-complementary) halfspaces of $G$. If all of edges in an $\Theta$-class $\theta_i$ are contracted, the resulting graph is called the ($\theta_i$-){\em contraction} of $G$. It is known that performing restrictions and contractions a finite number of times on $G$ results in a graph $G'$ that remains a partial cube, regardless of the order of the operations (see \cite{ckm20,m18}) and $G'$ is called  a {\em partial cube-minor} (abbreviated, {\em pc-minor}) of $G$. Note that possibly $G'=G$.  A pc-minor $G'$ of $G$ is {\em proper} if $G'\neq G$. A class $\mathscr{C}$ of partial cubes is called {\em pc-minor-closed} if  $G\in\mathscr{C}$ and $G'$ being a pc-minor of $G$ implies $G'\in\mathscr{C}$. For a (possibly infinite) set $X$ of partial cubes, let $\mathscr{F}(X)$ be the class of all partial cubes that do not contain any partial cube in $X$ as a  pc-minor. For a pc-minor-closed partial cube class $\mathscr{C}$, there exists a set $X$ of partial cubes such that $\mathscr{C}=\mathscr{F}(X)$ \cite{m18}. In this case, we say that the partial cubes in $X$ are the {\em forbidden pc-minors} of $\mathscr{C}$ or $\mathscr{C}$ is {\em $X$-pc-minor free}.  It is obvious that if  $G_1$ is a forbidden pc-minor of the class $\mathscr{C}$ and $G_1$ is a pc-minor of $G_2$, then $G_2$ is also a forbidden pc-minor of $\mathscr{C}$. 
For two sets of partial cubes $X$ and $Y$, it is clear by definition that $\mathscr{F}(X\cup Y)=\mathscr{F}(X)\cap\mathscr{F}(Y)$. For brevity, $\mathscr{F}(\{G_1, G_2,\cdots\})$ is written as $\mathscr{F}(G_1,G_2,\cdots)$. It was proved that the classes of median graphs (see Lemma 12.6 and Exercise 12.7 in \cite{hik11}) and daisy cubes (see \cite{km19,t20}) are pc-minor-closed. Consequently,  the class of simplex graphs is also pc-minor-closed by our first characterization. It is known that the class of median graphs is precisely $\mathscr{F}(Q_3^-,C_6)$, where $Q_3^-$ is the graph obtained by deleting one vertex from 3-cube $Q_3$ (see Proposition \ref{pro:ForbiddenMinorofMedianGraph} below). However, a similar result for daisy cubes remains unknown. In this paper, our second characterization of simplex graphs is presented: the $\{P_4\}$-pc-minor free median graphs, or equivalently, the partial cubes in $\mathscr{F}(P_4,Q_3^-,C_6)$.

Let $G$ be a partial cube. A $\Theta$-class $F_{uv}$ is  called {\em peripheral} if either $U_{uv}=W_{uv}$ or $U_{vu}=W_{vu}$. Our third characterization of simplex graphs states  that a graph is a simplex graph if and only if it is a median graph whose every $\Theta$-class is peripheral.

The {\em isometric dimension} of  a partial cube $G$, denoted by $\mathrm{idim}(G)$, is the smallest integer $n$ such that $G$ can be isometrically embedded into $Q_n$. This dimension coincides with the number of $\Theta$-classes of $G$ (see \cite{d73}). It is also known that any two adjacent edges in a partial cube belong to  different $\Theta$-classes (see Proposition \ref{pro:DWRelationonGeodesic} in the next section). Thus, $\mathrm{idim}(G)\geq \mathrm{deg}(v)$ for any vertex $v\in V(G)$, where $\mathrm{deg}(v)$ is the {\em degree} of $v$. Our fourth characterization of simplex graphs is that a graph $G$ is a simplex graph if and only if $G$ is a median graph containing a vertex $v\in V(G)$ such that $\mathrm{deg}(v)=\mathrm{idim}(G)$.

The main theorem of the present paper is stated as follows:
\begin{thm}\label{thm:MainResult}
Let $G$ be a finite median graph. The following statements are equivalent:
\begin{enumerate}
\item $G$ is a simplex graph;
\item $G$ is a daisy cube;
\item $G\in\mathscr{F}(P_4)$;
\item every $\Theta$-class of $G$ is peripheral;
\item there exists a vertex $v\in V(G)$ such that $\mathrm{deg}(v)=\mathrm{idim}(G)$.
\end{enumerate}
\end{thm}

The following corollary can be derived. Let $\mathscr{M}$ be the class of median graphs. Then:
\begin{pro}{\em\cite{ckm20}}\label{pro:ForbiddenMinorofMedianGraph}
$\mathscr{M}$ is pc-minor-closed. Moreover, $\mathscr{M}=\mathscr{F}(Q_3^-,C_6)$.
\end{pro}

As a corollary of Theorem \ref{thm:MainResult} and Proposition \ref{pro:ForbiddenMinorofMedianGraph}, we obtain:
\begin{cor}\label{cor:ForbiddenMinorofSimplexGraph}
Let $G$ be a partial cube. $G$ is a simplex graph if and only if $G\in\mathscr{F}(P_4,Q_3^-,C_6)$.
\end{cor}

Let $X$ be a set and $\Delta$ a family of subsets of $X$. $\Delta$ is called an {\em abstract simplicial complex} (abbreviated as ASC, also called an {\em independence system} in the context of matroids and greedoids) of $X$ if for any $S\in\Delta$ and $S'\subseteq S$, it holds that $S'\in\Delta$. The elements of an ASC $\Delta$ are called the {\em faces} (or {\em simplices}). The maximal faces, i.e., those faces in $\Delta$ that are not proper subsets of other faces, are called {\em facets} of $\Delta$. For example, let $\Delta$ be the family of cliques (resp. independent sets) of a finite graph $G$. Then $\Delta$ is an ASC of $V(G)$, called the {\em clique complex} (resp. {\em independence complex}) of $G$. It is clear that the clique complex of $G$ is equivalent to the independence complex of the complement of $G$. 

For any three subsets $A, B, C\subseteq X$, denote
\begin{equation*}
M(A, B, C):=(A\cap B)\cup (A\cap C)\cup (B\cap C).
\end{equation*}
Let $\Delta$ be a family of subsets of $X$. We say that $\Delta$ satisfies the {\em Median Property} if
\begin{equation}
\text{for any }D_1, D_2, D_3\in\Delta\text{, }M(D_1, D_2, D_3)\in\Delta.\tag{Median Property}
\end{equation}

Using the equivalence between Theorem \ref{thm:MainResult} (i) and (ii), we establish the following:

\begin{thm}\label{thm:CliqueComplex}
Let $X$ be a finite set and $\Delta$ an ASC of $X$. $\Delta$ can be represented as a clique complex (or equivalently, an independence complex) of a graph if and only if it satisfies the Median Property.
\end{thm}

\begin{proof}
Without loss of generality (W.l.o.g.), assume $X=[n]:=\{1,2,\cdots,n\}$. It is known that there exists a natural bijection $\pi$ from $B^n$ to $2^{[n]}$ ( the {\em power set} of $[n]$, which includes all subsets of $[n]$, including the empty set and $[n]$ itself): for any $v=a_1a_2\cdots a_n$, $a_i=1\Longleftrightarrow i\in\pi(v)$ ($1\leq i\leq n$). By the definitions, $\pi$ maps the vertex set of a daisy cube to an ASC of $[n]$, and the vertex set of a median graph to a family of subsets of $[n]$ satisfying the Median Property (see \cite{bk47}). By Theorem \ref{thm:MainResult} (i) and (ii), the theorem follows.
\end{proof}

Additionally, we say an ASC $\Delta$ satisfies the {\em Weak Median Property} if
\begin{equation}
\text{for any facets }D_1, D_2, D_3\in\Delta\text{, }M(D_1, D_2, D_3)\text{ is a face of }\Delta.\tag{Weak Median Property}
\end{equation}

In fact, for an ASC of a finite set, the Weak Median Property and the Median Property are equivalent.
\begin{pro}\label{pro:MPvsWMP}
For an ASC $\Delta$ of a finite set $X$, the Weak Median Property $\Longleftrightarrow$ the Median Property.
\end{pro}
\begin{proof} The implication `Median Property $\Longrightarrow$ Weak Median Property' is trivial. Now we prove `Weak Median Property $\Longrightarrow$ Median Property'.

For any $D_1, D_2, D_3\in\Delta$ (i.e., $D_1, D_2, D_3$ are faces of $\Delta$), let $E_i$ ($i=1,2,3$) be a facet of $\Delta$ such that $D_i\subseteq E_i$ ($E_i$ exists since $X$ is finite). It can be deduced that $M(D_1, D_2, D_3)\subseteq M(E_1, E_2, E_3)$. By the Weak Median Property, $M(E_1, E_2, E_3)$ is a face of $\Delta$. Since $\Delta$ is an ASC, $M(D_1, D_2, D_3)$ is also a face of $\Delta$. Thus,  the Median Property holds.
\end{proof}

\begin{rem}
Recently, Betre et al.\ \cite{bze24} proved that an ASC of a finite set can be represented as a clique complex of a graph if and only if it satisfies the Weak Median Property. Their proof utilized the equivalence between flag complexes and clique complexes. Combined with Proposition \ref{pro:MPvsWMP}, Theorem \ref{thm:CliqueComplex} re-establishes this result by using a graph-theoretical proof.
\end{rem}

We will introduce some terminology and propositions in the next section and prove Theorem \ref{thm:MainResult} in Section 3. Finally, we will conclude the paper and propose some future problems in Section 4.

\section{Preliminaries}

Throughout this paper, except where stated otherwise, all graphs we considered are undirected, finite, and simple. Let $G$ be a graph with vertex set $V(G)$  and edge set $E(G)$. For $u, v\in V(G)$, a shortest path from $u$ to $v$ is called a $u,v$-{\em geodesic}. For the geodesics of partial cubes, we have the following propositions:

\begin{pro}{\em\cite{ik00}}\label{pro:DWRelationonGeodesic}
Let $G$ be a partial cube containing a path $P$. $P$ is a geodesic in $G$ if and only if no two distinct edges on $P$ are in relation $\Theta$. In particular, two adjacent edges in a partial cube can't be in relation $\Theta$.
\end{pro}

\begin{pro}{\em\cite{xfx23}}\label{pro:DWRelationonTwoGeodesic}
Let $G$ be a partial cube, $u,v$ two vertices with $d(u,v)\geqslant 2$, and $P$, $P'$ two $u,v$-geodesics. Then, for every edge $e$ on $P$, there exists exactly one edge $e'$ on $P'$ such that $e'\,\Theta\, e$.
\end{pro}

Let $G$ be a partial cube. 
When we contract $G$ with respect to a $\Theta$-class $\theta_i$, the edges in $\theta_i$ are contracted (shortly, we say that $\theta_i$ is {\em contracted}). When we restrict $G$ with respect to a $\theta_i$, all edges are deleted except the ones whose both two endvertices are in the corresponding halfspace. For the proper pc-minors of partial cubes, we have

\begin{obs}\label{obv:ThetaClass}
Let $G$ be a partial cube with $e, f\in E(G)$ and $H$ its proper pc-minor. If $e$ and $f$ are neither contracted nor deleted in the contraction-restriction procedure from $G$ to $H$,  and $e', f'$ are their corresponding edges in $H$, respectively, then $e\,\Theta_G\,f\Longleftrightarrow e'\,\Theta_{H}\,f'$.
\end{obs}

Let $H$ be a subgraph of a graph $G$. $H$ is called a {\em convex subgraph} of $G$ if, for any $u,v\in V(H)$, all $u,v$-geodesics are contained in $H$. It is obvious that the convex subgraphs are isometric. Denote $\partial\, H=\{uv\in E(G)|u\in V(H),v\not\in V(H)\}$. Regarding convex subgraphs of bipartite graphs, we have
\begin{pro}{\em\cite{ik98}}\label{pro:ConvexityLemma}
An induced connected subgraph $H$ of a bipartite graph $G$ is convex if and only if no edge of $\partial\, H$ is in relation $\Theta$ to an edge in $H$.
\end{pro}

The following proposition can be deduced from Proposition \ref{pro:ConvexityLemma}.
\begin{pro}{\em\cite{ak16,b89,c86}}\label{pro:ConvexSubgraph}
Let $G$ be a partial cube and $H$ a subgraph of $G$. $H$ is a convex subgraph of $G$ if and only if $H$ is a restriction of $G$.
\end{pro}

Let $G$ be a graph. For a subset $S$ of vertex set $V(G)$, the subgraph induced by $S$ is denoted by $G[S]$.  For the median graphs, it is known that there are many equivalent characterizations. The most well-known is the following proposition:

\begin{pro}{\em\cite{ik00}}\label{pro:UisConvex}
A graph $G$ is a median graph if and only if $G$ is a partial cube and $G[U_{uv}]$ and $G[U_{vu}]$ are convex in $G$ for every $uv\in E(G)$.
\end{pro}

Let $G$ be a partial cube, $\theta_i,\theta_j$ two $\Theta$-classes of $G$. We say $\theta_i$ and $\theta_j$ {\em cross} if all of $\rho_{i}^0(G)\cap \rho_{j}^0(G)$,  $\rho_{i}^0(G)\cap \rho_{j}^1(G)$,  $\rho_{i}^1(G)\cap \rho_{j}^0(G)$ and  $\rho_{i}^1(G)\cap \rho_{j}^1(G)$ are not empty. Another equivalent expression of crossing relation is as follows:

\begin{pro}{\em\cite{km02}}\label{pro:CrossinIsometricCycle}
Let $G$ be a partial cube, $\theta_i,\theta_j$ two $\Theta$-classes. Then $\theta_i$ and $\theta_j$ cross if and only if they occur on an isometric cycle $C$, i.e., $E(C)\cap\theta_i\neq\emptyset$, $E(C)\cap\theta_j\neq\emptyset$.
\end{pro}

Assume $C$ is an isometric cycle of length $2l$ in a partial cube $G$. Then two edges $e,f\in E(C)$ are in the relation $\Theta$ if and only if they are antipodal edges of $C$ (see \cite{m09}). Therefore, there are exactly $l$ $\Theta$-classes occurring on $C$. If we contract $k$ ($k\leq l-2$) $\Theta$-classes of them, the resulting cycle remains isometric. Combined with Observation \ref{obv:ThetaClass} and Proposition \ref{pro:CrossinIsometricCycle}, 
the crossing relation remains preserved under the contractions. That is:
\begin{pro}\label{pro:CrossingunderContraction}
Let $G$ be a partial cube with $uv, xy\in E(G)$, where $uv$ and $xy$ are not in relation $\Theta_G$, and let $H$ be a pc-minor obtained by taking a sequence of contractions on $G$. If the edges $uv$ and $xy$ are not contracted in the contraction procedure from $G$ to $H$, and $u',v',x',y'$ are corresponding vertices of $u,v,x,y$ in $H$, respectively, then $F^G_{uv}$ and $F^G_{xy}$ cross in $G$ if and only if $F^H_{u'v'}$ and $F^H_{x'y'}$ cross in $H$.
\end{pro}

Let $G$, $\theta_i$ and $\theta_j$ be defined as Proposition \ref{pro:CrossinIsometricCycle}. A 4-cycle $uvwxu$ is called to be {\em $\theta_i,\theta_j$-alternating} if $uv, xw\in\theta_i$ and $ux, vw\in\theta_j$. For the median graphs, the following proposition is stronger than Proposition \ref{pro:CrossinIsometricCycle}.

\begin{pro}{\em\cite{xfx24}}\label{pro:Crossin4Cycle}
Let $G$ be a median graph, $\theta_i, \theta_j$ two $\Theta$-classes. Then $\theta_i$ and $\theta_j$ cross if and only if there exists a $\theta_i, \theta_j$-alternating 4-cycle.
\end{pro}

The {\em crossing graph} $G^{\#}$ of the partial cube $G$ is the graph whose vertices are corresponding to the $\Theta$-classes of $G$, and $\theta_i$ and $\theta_j$ are adjacent in $G^{\#}$ if and only if they cross in $G$. The crossing graph was introduced by Bandelt and Dress \cite{bd92} under the name of incompatibility graph and extensively studied by Klav\v zar and Mulder \cite{km02}. For the relation between the crossing graphs and the simplex graphs, Klav\v zar and Mulder proved:

\begin{pro}{\em\cite{km02}}\label{pro:SimplexvsCrossing}
For any graph $G$, $(S(G))^{\#}=G$.
\end{pro}

Let $G$ be a graph. Denote the number of induced $i$-cubes of $G$ by $\alpha_i(G)$ for $i\geq 0$. The polynomial

\begin{equation*}
C(G,x):=\sum_{i\geq 0}\alpha_i(G)x^i
\end{equation*}
is called the {\em cube polynomial} of $G$ \cite{bks03}.

Let $a_i(G)$ $(i\geq 1)$ be the number of clique of size $i$ in $G$, and let $a_0(G)=1$. The {\em clique polynomial} of a graph $G$, introduced by Hoede and Li \cite{hl94},  is defined as follows:

\begin{equation*}
Cl(G,x):=\sum_{i\geq 0}a_i(G)x^i.
\end{equation*}

Xie et al. gave a relation between the cube polynomials of median graphs and the clique polynomials of their crossing graphs as follows:

\begin{pro}{\em\cite{xfx24}}\label{pro:=forMedianGraph}
Let $G$ be a median graph and $G\neq K_1$. Then $C(G,x)=Cl(G^{\#},x+1)$.
\end{pro}

Combined with the fact that $\alpha_0(G)=|V(G)|$ for any graph $G$, the following proposition is an immediate corollary of Proposition \ref{pro:=forMedianGraph}:

\begin{pro}\label{pro:VertexNumber}
Let $G, G'$ be two median graphs and $G, G'\neq K_1$. If $G^{\#}\cong (G')^{\#}$, then $C(G,x)=C(G',x)$. In particular, $|V(G)|=|V(G')|$.
\end{pro}

Finally, for daisy cubes, we have
\begin{pro}{\em (see Proposition 2.2 in \cite{km19} and Proposition 2.4 in \cite{t20})}\label{pro:DaisyCubeisPCMinorClosed}
The class of daisy cubes is pc-minor-closed.
\end{pro}

\section{Proof of Theorem \ref{thm:MainResult}}

We now proceed to prove Theorem \ref{thm:MainResult}.

\begin{proof}[Proof of Theorem \ref{thm:MainResult}]When $G=K_1$, $G$ is a simplex graph of a empty graph, i.e., the vertex set is empty, and therefore Theorem \ref{thm:MainResult} is trivial in this case. Now, we consider the case of $G\neq K_1$, in other words, $\mathrm{idim}(G)\geq 1$.

(i) $\Longrightarrow$ (ii). W.l.o.g., assume $G=S(H)$ for a graph $H$ with $V(H)=[n]$. By definition, each vertex in $G$ corresponds to a clique of $H$, which is a subset of $[n]$. Recall the bijection $\pi:B^n\to 2^{[n]}$ used in the proof of Theorem \ref{thm:CliqueComplex}. Then $\pi^{-1}$ maps a clique $C$ of $H$ to a binary string  $v=\pi^{-1}(C)=a_1a_2\cdots a_n$: $i\in C\Longleftrightarrow a_i=1$ for $i\in [n]$. Since the clique complex of $H$ is an ASC, $\pi^{-1}$ induces an isomorphism from $G$ to a partial cube $G'$ such that if $v\in V(G')$, then for any $v'\leq v$, $v'\in V(G')$. By definition, $G'$ is a daisy cube, and therefore  $G$ is also a daisy cube.

(ii) $\Longrightarrow$ (iii). By Propositions \ref{pro:ForbiddenMinorofMedianGraph} and \ref{pro:DaisyCubeisPCMinorClosed}, the classes of median graphs and daisy cubes are pc-minor-closed. Thus,  we only need to show that $P_4$ is not a daisy cube.

Assume, for contradiction, that $P_4:=uvwx$ is a daisy cube. By Proposition \ref{pro:DWRelationonGeodesic}, any two different edges of $P_4$ are not in relation $\Theta$, so $\mathrm{idim}(P_4)=3$. W.l.o.g., assume $uv\in\theta_1$, $vw\in\theta_2$, $wx\in\theta_3$, where $\theta_1=\{ab|a=0a_2a_3, b=1a_2a_3\}$, $\theta_2=\{ab|a=a_10a_3, b=a_11a_3\}$, and $\theta_3=\{ab|a=a_1a_20, b=a_1a_21\}$. By the definition of daisy cubes, one vertex in $V(P_4)$ must be labelled 000. If $u=000$, then $v=100$, $w=110$, and $x=111$. However,  the definition of daisy cubes requires $010\in V(P_4)$, which is a contradiction. Similarly, if $v=000$, then $u=100$, $w=010$, and $x=011$, and the definition of daisy cubes requires $001\in V(P_4)$, another contradiction. The same reasoning applies to the case $w=000$ or $x=000$ due to the symmetry of $P_4$. Thus, $P_4$ is not a daisy cube. 

(iii) $\Longrightarrow$ (iv). Assume, for contradiction, that the $\Theta$-class $F^G_{uv}$ is not peripheral; that is, $W^G_{uv}-U^G_{uv}\neq\emptyset$ and $W^G_{vu}-U^G_{vu}\neq\emptyset$. Choose $a\in W^G_{uv}-U^G_{uv}$, $b\in W^G_{vu}-U^G_{vu}$ such that $d(a,b)$ is minimal. Then $a$ (resp. $b$) is adjacent to a vertex in $U^G_{uv}$ (resp. $U^G_{vu}$), otherwise we could select $a'\in W^G_{uv}-U^G_{uv}$ and $b'\in W^G_{vu}-U^G_{vu}$ such that $d(a',b')<d(a,b)$, contradicting the minimality of $d(a, b)$. Since $a\in W^G_{uv}$, $b\in W^G_{vu}$ and $F^G_{uv}$ is an edge cutset of $G$, there exists an edge in $F^G_{uv}$ on any $a,b$-geodesic. Furthermore,  since $a\not\in U^G_{uv}$ and $b\not\in U^G_{vu}$, $d(a,b)\geq 3$. By Proposition \ref{pro:UisConvex}, $G[U^G_{uv}]$ is convex, so we can assume that $P:=au_1u_2\cdots u_kv_kb$ is an $a,b$-geodesic where $u_1,u_2,\cdots,u_k\in U^G_{uv}$ and $v_k\in U^G_{vu}$. We consider two cases.

{\bf Case 1.} $k=1$.

In this case, $P=au_1v_1b$ is a $P_4$. If $P$ is a convex subgraph, then it is a pc-minor of $G$ by Proposition \ref{pro:ConvexSubgraph}, contradicting (iii). If $P$ is not convex, then there exists another $a,b$-geodesic $\tilde{P}:=a\tilde{u}_1\tilde{v}_1b$ such that $\tilde{P}\neq P$. W.l.o.g., assume $\tilde{u}_1\neq u_1$. By Proposition \ref{pro:DWRelationonTwoGeodesic}, there is an edge in $E(\tilde{P})\cap F^G_{uv}$. Since $a\not\in U^G_{uv}$ (resp. $b\not\in U^G_{vu}$), $a\tilde{u}_1\not\in F^G_{uv}$ (resp. $\tilde{v}_1b\not\in F^G_{uv}$). Thus, $\tilde{u}_1\tilde{v}_1\in F^G_{uv}$, implying $\tilde{u}_1\in U^G_{uv}$. Since $\tilde{u}_1\neq u_1$ and $G$ is bipartite, the path $\tilde{u}_1au_1$ is a path of length 2, and further, a $u_1,\tilde{u}_1$-geodesic. Since $u_1,\tilde{u}_1\in U^G_{uv}$, by proposition \ref{pro:UisConvex}, $a\in U^G_{uv}$, contradicting our assumption.

{\bf Case 2.} $k\geq 2$.

In this case, let $v_i$ ($1\leq i\leq k-1$) be another endvertex of the edge $u_iv_i$ such that $u_iv_i\in F^G_{uv}$. Contract the $\Theta$-classes $F^G_{u_1u_2},F^G_{u_2u_3},\cdots,F^G_{u_{k-1}u_k}$, and let the resulting graph be $G'$. By Proposition \ref{pro:DWRelationonGeodesic}, the $\Theta$-classes $F^G_{au_1}$ and $F^G_{uv}$ are not contracted. Since $G$ is a median graph, $G'$ is also a median graph by Proposition \ref{pro:ForbiddenMinorofMedianGraph}. Denote the corresponding vertex of $u_1,u_2,\cdots,u_k$ in $G'$ by $u'$, that of $v_1,v_2,\cdots,v_k$ by $v'$,  $a$ by $a'$, and $b$ by $b'$ (see Fig. \ref{fig:abgeodesic}). We now prove that $a'\not\in U^{G'}_{u'v'}$.

\begin{figure}[!htbp]
\centering
\scalebox{0.5}[0.5]{\includegraphics{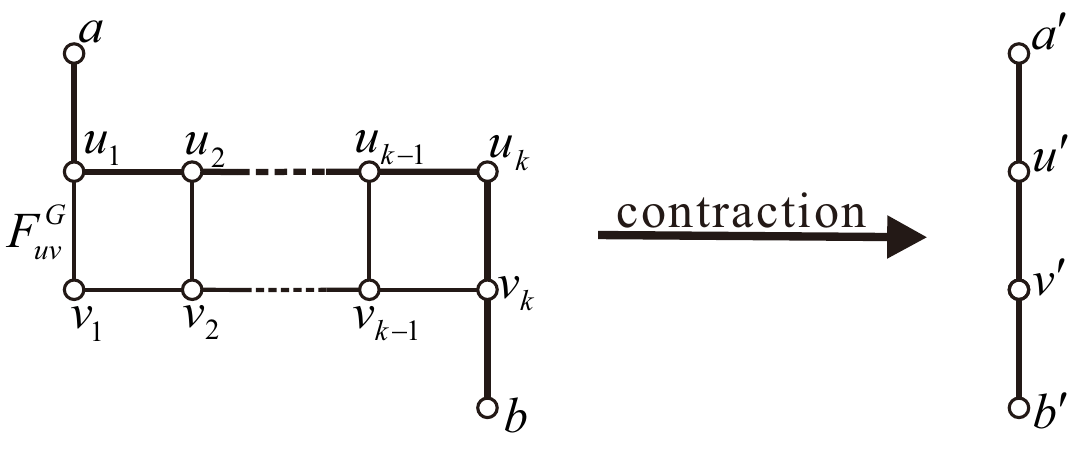}}
\caption{Illustration for Case 2 in the proof of (iii)$\Longrightarrow$(iv).}{\label{fig:abgeodesic}}
\end{figure}

By contradiction, assume $a'\in U^{G'}_{u'v'}$. Then the edges $a'u'$ and $u'v'$ are in an $F^{G'}_{a'u'},F^{G'}_{u'v'}$-alternating 4-cycle, so $F^{G'}_{a'u'}$ and $F^{G'}_{u'v'}$ cross in $G'$ by Proposition \ref{pro:Crossin4Cycle}. By Proposition \ref{pro:CrossingunderContraction}, $F^{G}_{au_1}$ and $F^{G}_{uv}$ cross in $G$. Thus, by Proposition \ref{pro:Crossin4Cycle}, there is an $F^{G}_{au_1},F^{G}_{uv}$-alternating 4-cycle in $G$, which implies there is an edge $e\in E(G[U^G_{uv}])$ such that $au_1\,\Theta_G\,e$. By Propositions \ref{pro:ConvexityLemma} and \ref{pro:UisConvex}, $a\in U^G_{uv}$, a contradiction with the hypothesis.

Similarly, $b'\not\in U^{G'}_{v'u'}$. We can discuss the 3-path $P':=a'u'v'b'$ in $G'$ similar to Case 1. It can be deduced that $P'$ is a pc-minor of $G'$, so is a pc-minor of $G$, a contradiction with (iii).

(iv) $\Longrightarrow$ (v). Assume $\mathrm{idim}(G)=n$. Let us assign a labelling of $B^n$ to the vertices of $G$, similar to the one presented in Algorithm 1 in \cite{t20}. Let $F_{u_1v_1},F_{u_2v_2},\cdots,F_{u_nv_n}$ be the $n$ different $\Theta$-classes of $G$. Since every $\Theta$-class is peripheral, for any $i\in [n]$, $W_{u_iv_i}=U_{u_iv_i}$ or $W_{v_iu_i}=U_{v_iu_i}$. For each $i\in [n]$, if $W_{u_iv_i}=U_{u_iv_i}$, set the $i$-th coordinate of the label of $v\in W_{v_iu_i}$ to 0 and the one of $v\in W_{u_iv_i}$ to 1; otherwise, set the $i$-th coordinate of the label of $v\in W_{u_iv_i}$ to 0 and the one of $v\in W_{v_iu_i}$ to 1. Since we assign different coordinates of the labels each time, this labelling is well-defined. Under this labelling, for any $i\in [n]$, if $a_1\cdots a_{i-1}1a_{i+1}\cdots a_n\in V(G)$, then  $a_1\cdots a_{i-1}1a_{i+1}\cdots a_n\in U_{u_iv_i}$ or $U_{v_iu_i}$, so $a_1\cdots a_{i-1}0a_{i+1}\cdots a_n\in V(G)$. By analogy, we can deduce that $0^n:=00\cdots 0\in V(G)$. Furthermore, If $v\in V(G)$, then for any $u\leq v$, $u\in V(G)$. Since $\mathrm{idim}(G)=n$, the halfspace $\{a=a_1a_2\cdots a_n\in V(G)|a_i=1\}$ is not empty for any $i\in[n]$. Let $w_i$ be a vertex in this halfspace. 
Denote $x_i:=00\cdots 010\cdots 0$ where the $i$-th coordinate of its label is 1 and all of other coordinates are 0. Then $x_i\leq w_i$ and further $x_i\in V(G)$. In addition, $x_i$ ($i\in [n]$) is adjacent to $0^n$. Thus $n\leq\deg (0^n)\leq\mathrm{idim}(G)=n$. Therefore $\deg (0^n)=\mathrm{idim}(G)$.

(v) $\Longrightarrow$ (i). Assume $\mathrm{idim}(G)=n$. There exists a vertex such that its degree equals  $n$. W.l.o.g., let us label this vertex as $0^n$. Denote $L_k:=\{v\in V(G)|d(v,0^n)=k\}$ for $k\geq 1$. Then $L_1=\{x_1,x_2,\cdots,x_n\}$, where $x_i=00\cdots 010\cdots 0$ (only the $i$-th coordinate is 1, and all of other coordinates are 0). Also, $L_2$ is a subset of the set $\{x_{ij}=a_1a_2\cdots a_n|a_l=1\mbox{ if }l=i,j\mbox{; otherwise }a_l=0,\mbox{ for }1\leq i<j\leq n\}$. Let $H$ be a graph with $V(H)=[n]$, where $i, j\in V(H)$ are adjacent if and only if $x_{ij}\in L_2$. Now, we prove that $G\cong S(H)$.

Let $\pi^{-1}$ be the bijection in the proof of (i) $\Longrightarrow$ (ii). We prove that for any clique $C$ of $H$, $\pi^{-1}(C)\in V(G)$ by induction on $|C|$. When $|C|=1$ or 2, it holds by the definition of $H$. Assume that for any clique $C$ of size $k-1$ ($k\geq 3$), $\pi^{-1}(C)\in V(G)$. Now, let $|C|=k$. For any three different vertices $r, s, t\in C$, $C-\{r\},C-\{s\},C-\{t\}$ are the cliques of $H$ of size $k-1$. By the inductive hypothesis, $\pi^{-1}(C-\{r\}),\pi^{-1}(C-\{s\}),\pi^{-1}(C-\{t\})\in V(G)$.  We observe that the Hamming distances $d_{Q_n}(\pi^{-1}(C-\{r\}),\pi^{-1}(C-\{s\}))=d_{Q_n}(\pi^{-1}(C-\{r\}),\pi^{-1}(C-\{t\}))=d_{Q_n}(\pi^{-1}(C-\{s\}),\pi^{-1}(C-\{t\}))=2$, and further $\pi^{-1}(C)$ is the median of $\pi^{-1}(C-\{r\})$, $\pi^{-1}(C-\{s\})$ and $\pi^{-1}(C-\{t\})$. Since $G$ is a median graph, $\pi^{-1}(C)\in V(G)$.

Thus, $\pi^{-1}$ is an isomorphism from $S(H)$ to an induced subgraph of $G$. Let us consider the crossing graph of $G$, $G^{\#}$. Recall that $\theta_i=F^G_{0^nx_i}=\{ab\in E(G)|a=a_1\cdots a_{i-1}0a_{i+1}\cdots a_n,b=a_1\cdots a_{i-1}1a_{i+1}\cdots a_n\}$ for $i\in[n]$. We prove that:

{\bf Claim 1.} For $1\leq i<j\leq n$, $\theta_i$ and $\theta_j$ cross if and only if $x_{ij}\in L_2$.

{\em Sufficiency.} Since $x_{ij}\in L_2$, $0^nx_ix_{ij}x_j0^n$ is a $\theta_i,\theta_j$-alternating 4-cycle. By Proposition \ref{pro:Crossin4Cycle}, $\theta_i$ and $\theta_j$ cross.

{\em Necessity.} Since $\theta_i$ and $\theta_j$ cross, by Proposition \ref{pro:Crossin4Cycle}, there exists an edge $e\in\theta_j\cap E(G[U_{0^nx_i}])$. Since $0^nx_j\,\Theta\,e$, by Propositions \ref{pro:ConvexityLemma} and \ref{pro:UisConvex}, $x_j \in U_{0^nx_i}$, and further $x_{ij}\in L_2$.

By Claim 1, we deduce $G^{\#}=H$. Combined with Propositions \ref{pro:SimplexvsCrossing} and \ref{pro:VertexNumber}, we have $|V(G)|=|V(S(H))|$. Combined with the fact that $\pi^{-1}$ is an isomorphism from $S(H)$ to an induced subgraph of $G$, $G\cong S(H)$.
\end{proof}

\section{Conclusions and the problem of characterizing the daisy cubes}

In this paper, we have characterized simplex graphs through four equivalent descriptions, which are respectively associated with the daisy cubes, forbidden pc-minors, peripheral $\Theta$-classes, and vertices whose degree equals the isometric dimension. Additionally, following Betre et al.\ \cite{bze24}, we re-proved that an abstract simplicial complex of a finite set can be represented as the clique complex of a graph if and only if it satisfies the Median Property (or equivalently, the Weak Median Property), using the graph-theoretical method.

Taranenko \cite{t20} provided a characterization of daisy cubes, stating that a connected graph $G$ is a daisy cube if and only if it can be obtained from the one-vertex graph by a sequence of special expansions, known as {\em $\leq$-expansions}. However, this characterization relies on the partial order `$\leq$' on $B^n$, which depends on the labelling of binary strings assigned to the vertices. This raises the question: Can there exist a structural characterization that is independent of labelling? By Proposition \ref{pro:DaisyCubeisPCMinorClosed}, the class of daisy cubes is pc-minor-closed. This naturally leads to the following problem:
\begin{que}\label{que:daisycubes}
Characterize the daisy cubes by using their forbidden pc-minors.
\end{que}

If a partial cube $G$ is not a daisy cube but all of its proper pc-minors are, then $G$ is called a {\em minimal forbidden pc-minor} of daisy cubes. To characterize daisy cubes via forbidden pc-minors, the objective is to identify all minimal forbidden pc-minors of daisy cubes. For notation, similar to $0^n$, denote $1^n:=11\cdots 1\in B^n$. Define $Q_n^{-+}$ as the graph obtained by deleting $1^n$ from $Q_n$ and adding a pendant vertex attached to the vertex $00\cdots 01$, and $Q_n^{--}$ as the graph obtained by deleting $0^n$ and $1^n$ from $Q_n$. Examples of $Q_3^{-+}$ and $Q_4^{--}$ are shown in Fig. \ref{fig:ForbiddenPCMinor1}. It can be verified that $Q_n^{-+}$ ($n\geq 2$) and $Q_n^{--}$ ($n\geq 3$) are minimal forbidden pc-minors of daisy cubes. Specifically, $Q_2^{-+}$ is the 3-path $P_4$ and $Q_3^{--}$ is the 6-cycle $C_6$.

\begin{figure}[!htbp]
\centering
\setlength{\unitlength}{1mm}
\begin{picture}(60,45)
\put(0,0){\scalebox{0.5}[0.5]{\includegraphics{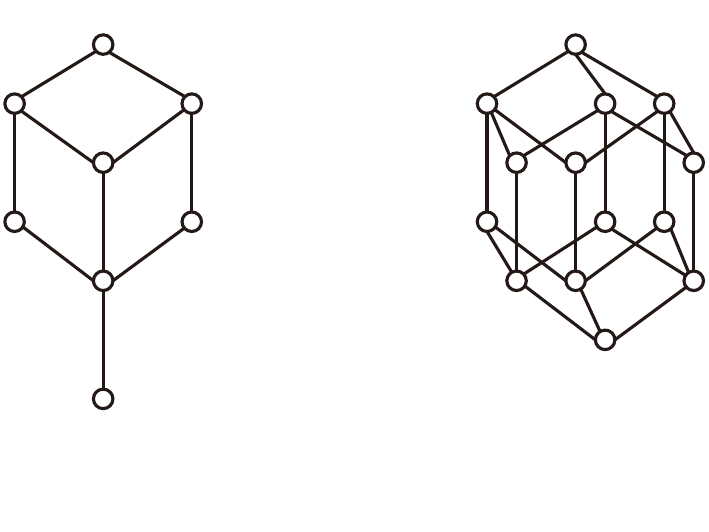}}}
\put(6,5){$Q_3^{-+}$}
\put(47,5){$Q_4^{--}$}
\end{picture}
\caption{$Q_3^{-+}$ and $Q_4^{--}$.}{\label{fig:ForbiddenPCMinor1}}
\end{figure}

In addition, let us give some more examples of minimal forbidden pc-minors of daisy cubes. The {\em Cartesian product} of graphs $G_1$ and $G_2$, denoted by $G_1 \Box G_2$, is the graph with vertex set $V(G_1)\times V (G_2)$ and two vertices $(u, v),(u', v')$ being adjacent if and only if either $u=u'$ and $vv'\in E(G_2)$ or $uu'\in E(G_1)$ and $v=v'$. Denote $H_n:=P_3\Box Q_{n-2}$ ($n\geq 2$), i.e., $V(H_n)=\{v=a_1a_2\cdots a_{n-1}|a_1=0,1,2\mbox{ and }a_i=0,1\mbox{ for }2\leq i\leq n-1\}$ and two vertices are adjacent if and only if exactly one coordinate of their corresponding strings differs by 1 while the other coordinates are the same. Let $H_n^{--}$ be the graph obtained by deleting $00\cdots 0$ and $21\cdots 1$ from $H_n$. Then $H_n^{--}$ ($n\geq 4$) is also a minimal forbidden pc-minor of daisy cubes. Example of $H_4^{--}$ is shown in Fig. \ref{fig:ForbiddenPCMinor2}.

\begin{figure}[!htbp]
\centering\setlength{\unitlength}{1mm}
\begin{picture}(32.5,27.5)
\put(0,0){\scalebox{0.5}[0.5]{\includegraphics{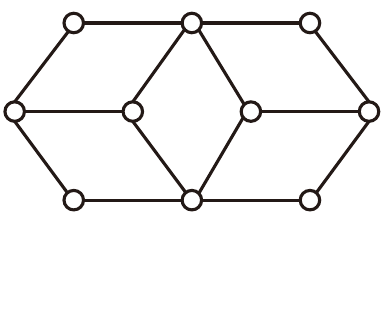}}}
\put(14,3){$H^{--}_4$}
\end{picture}
\caption{The partial cube $H^{--}_4$.}{\label{fig:ForbiddenPCMinor2}}
\end{figure}

Identifying other minimal forbidden pc-minors of daisy cubes remains a significant challenge and  a central focus for future research.

\vskip 0.2 cm
\noindent{\bf Acknowledgements:} This work is partially supported by National Natural Science Foundation of China (Grants No. 12071194, 11571155).


\begin{thebibliography}{99}
\parskip -0.1cm

\bibitem{ak16}M. Albenque, K. Knauer, Convexity in partial cubes: the hull number, Discrete Math., 339 (2016) 866--876.

\bibitem{b89}H.-J. Bandelt, Graphs with intrinsic $S_3$ convexities, J. Graph Theory, 13 (1989) 215--338.

\bibitem{bd92}H.-J. Bandelt, A. W. M. Dress, A canonical decomposition theory for metrics on a finite set, Adv. Math., 92 (1992) 47--105.

\bibitem{bv87}H.-J. Bandelt, M. van de Vel, A fixed cube theorem for median graphs, Discrete Math., 62 (1987) 129--137.

\bibitem{bv89}H.-J. Bandelt, M. van de Vel, Embedding topological median algebras in products of dendrons, Proc. London Math. Soc., s3-58 (3) (1989) 439--453.


\bibitem{bze24}K. H. Betre, Y. X. Zhang, C. Edmond, Pure simplicial and clique complexes with a fixed number of facets, 2024, arXiv: 2411.12945v1.

\bibitem{bk47}G. Birkhoff, S. A. Kiss, A ternary operation in distributive lattices, Bull. Amer. Math. Soc., 53 (1947) 749--752.

\bibitem{bks03}B. Bre\v sar, S. Klav\v zar, R. \v Skrekovski, The cube polynomial and its derivatives: the case of median graphs, Electron. J. Combin., 10 (2003) \#R3.

\bibitem{c86}V. Chepoi, $d$-convex sets in graphs (Dissertation), Moldova State University, Chi\c{s}in\v au, 1986, (in Russian).

\bibitem{ckm20}V. Chepoi, K. Knauer, T. Marc, Hypercellular graphs: Partial cubes without $Q_3^-$ as partial cube minor, Discrete Math., 343 (2020) 111678.

\bibitem{d73}D. \v{Z}. Djokovi\'c, Distance preserving subgraphs of hypercubes, J. Combin. Theory Ser. B, 14 (1973) 263--267.

\bibitem{hik11}R. Hammack, W. Imrich, S. Klav\v{z}ar, Handbook of product graphs, Boca Raton, CRC press, 2011.

\bibitem{hl94}C. Hoede, X. Li, Clique polynomials and independent set polynomials of graphs, Discrete Math., 125 (1994) 219--228.

\bibitem{ik98}W. Imrich, S. Klav\v zar, A convexity lemma and expansion procedures for bipartite graphs, European J. Combin., 19 (1998) 677--685.

\bibitem{ik00}W. Imrich, S. Klav\v zar, Product Graphs: Structure and Recognition, John Wiley \& Sons, New York, USA, 2000.

\bibitem{k09}A. Kirlangic, The rupture degree and gear graphs, Bull. Malays. Math. Sci. Soc., 32 (2009) 31--36.

\bibitem{k13}S. Klav\v zar, Structure of Fibonacci cubes: a survey, J. Comb. Optim., 25 (2013) 505--522.

\bibitem{km19}S. Klav\v zar, M. Mollard, Daisy cubes and distance cube polynomial, European J. Combin., 80 (2019) 214--223.

\bibitem{km99}S. Klav\v zar, H. M. Mulder, Median graphs: characterizations, location theory and related structures, J. Combin. Math. Combin. Comput., 30 (1999) 103--127.

\bibitem{km02}S. Klav\v zar, H. M. Mulder, Partial cubes and crossing graphs, SIAM J. Discrete Math., 15 (2002) 235--251.

\bibitem{m18}T. Marc, Cycling in hypercubes, PhD thesis, Univ. of Ljubljana, 2018.

\bibitem{m09}M. Massow, Linear extension graphs and linear extension diameter, Dissertation, Technische Universit\"at Berlin, 2009.

\bibitem{m78}H. M. Mulder, The structure of median graphs, Discrete Math., 24 (1978) 197--204.

\bibitem{m80a}H. M. Mulder, $n$-cubes and median graphs, J. Graph Theory, 4 (1980) 107--110.

\bibitem{m80b}H. M. Mulder, The interval function of a graph, PhD thesis, Vrije Universiteit Amsterdam, 1980.

\bibitem{m90}H. M. Mulder, The expansion procedure for graphs, in: R. Bodendiek (Ed.), Contemporary Methods in Graph Theory, Wissenschaftsverlag, Mannheim, 1990, pp. 459--477.

\bibitem{m11}H. M. Mulder, Median graphs. A structure theory, in: H. Kaul, H.M. Mulder (Eds.), Advances in Interdisciplinary Applied Discrete Mathematics, vol. 11, World Scientific, Singapore, 2011, pp. 93--125.


\bibitem{mcs01} E. Munarini, C. Perelli Cippo, N. Zagaglia Salvi, On the Lucas cubes, Fibonacci Q., 39 (2001) 12--21.

\bibitem{t13}A. Taranenko, A new characterization and a recognition algorithm of Lucas cubes, Discrete Math. Theor. Comput. Sci., 15 (2013) 31--39.

\bibitem{t20}A. Taranenko, Daisy cubes: A characterization and a generalization, European J. Combin., 85 (2020) 103058.

\bibitem{w84}P. M. Winkler, Isometric embedding in products of complete graphs, Discrete Appl. Math., 7 (1984) 221--225.

\bibitem{xfx23}Y.-T. Xie, Y.-D. Feng, S.-J. Xu, Characterization of 2-arc-transitive partial cubes, Discrete Math., 346 (2023) 113190.

\bibitem{xfx24} Y.-T. Xie, Y.-D. Feng, S.-J. Xu, A relation between the cube polynomials of partial cubes and the clique polynomials of their crossing graphs, J. Graph Theory, 106 (2024) 907--922.
\end{thebibliography}
\end{document}